\documentclass[11pt]{amsart}
\usepackage{amssymb,amscd,amsthm,verbatim}
\usepackage{amsbsy}
\usepackage{latexsym}

\date{\today}

\newcommand{\Z}{{\mathbb Z}}
\newcommand{\R}{{\mathbb R}}

\newcommand{\N}{{\mathbb N}}

\newcommand{\GL}{{\mathrm{GL}}}
\newcommand{\SL}{{\mathrm{SL}}}

\newtheorem{theorem}{Theorem}[section]
\newtheorem{lemma}[theorem]{Lemma}

\theoremstyle{definition}

\newtheorem*{definition}{Definition}
\newtheorem*{question}{Question}

\newtheorem{remark}[theorem]{Remark}
\theoremstyle{plain}

\allowdisplaybreaks \numberwithin{equation}{section}

\DeclareMathOperator{\supp}{supp}

\begin{document}

\title[Dense Phenomena for Ergodic Schr\"odinger Operators]{Dense Phenomena for Ergodic Schr\"odinger Operators: I. Spectrum, Integrated Density of States, and Lyapunov Exponent}

\author[A.\ Avila]{Artur Avila}

\address{Institut f\"ur Mathematik, Universit\"at Z\"urich, Winterthurerstrasse 190, 8057 Z\"urich, Switzerland and IMPA, Estrada D. Castorina 110, Jardim Bot\^anico, 22460-320 Rio de Janeiro, Brazil}

\email{artur.avila@math.uzh.ch}

\author[D.\ Damanik]{David Damanik}

\address{Department of Mathematics, Rice University, Houston, TX~77005, USA}

\email{damanik@rice.edu}

\thanks{D.\ D.\ was supported in part by NSF grants DMS--2054752 and DMS--2349919}


\begin{abstract}
We consider Schr\"odinger operators in $\ell^2(\Z)$ whose potentials are defined via continuous sampling along the orbits of a homeomorphism on a compact metric space. We show that for each non-atomic ergodic measure $\mu$, there is a dense set of sampling functions such that the associated almost sure spectrum has finitely many gaps, the integrated density of states is smooth, and the Lyapunov exponent is smooth and positive. As a byproduct we answer a question of Walters about the existence of non-uniform $\SL(2,\R)$ cocycles in the affirmative.
\end{abstract}

\maketitle

\section{Introduction}

\subsection{Context}

The Schr\"odinger equation $i \partial_t \psi = H \psi$ describes the time evolution of a quantum mechanical system. Here, $H$ denotes the Schr\"odinger operator, which takes the
form $H = -\Delta + V$ and acts as a self-adjoint operator in the underlying Hilbert space, which is usually given by $L^2(\R^d)$ or $\ell^2(\Z^d)$. The potential $V$ is a real-valued function defined on $\R^d$ or $\Z^d$ and it models the environment the quantum state is evolving in.

The time evolution given by the Schr\"odinger equation is often studied via time-independent methods by relating it to spectral properties of the Schr\"odinger operator. Invoking the spectral theorem for self-adjoint operators, the type of the spectral measure of the initial state $\psi(0)$ of the evolution gives information about the behavior of $\psi(t) = e^{-itH} \psi(0)$ as $|t| \to \infty$. Very roughly speaking, the more regular this spectral measure is, the faster the state will spread out in space. The classical correspondence in this spirit is given by the RAGE theorem, but more refined results are available by now. As a consequence, one is generally interested in the decomposition of spectral measures into their absolutely continuous, singular continuous, and pure point parts, and a determination of which parts are non-trivial.
 
The investigation of Schr\"odinger operators with potentials displaying a spatially stationary structure has been an area of intense activity in the past several decades, with significant advances leading to an in-depth understanding of their spectral properties. Specific examples of potentials with a spatially stationary structure are given by periodic potentials, almost periodic potentials, and random potentials.

The theory has been developed in greater detail in the case of one spatial dimension (i.e., when considering the Schr\"odinger evolution in $L^2(\R^d)$ or $\ell^2(\Z^d)$ with $d = 1$). This is especially true in the almost periodic case, but also in the random case our understanding is far more complete in one dimension. 

The class of ergodic Schr\"odinger operators provides a convenient umbrella framework for the investigation of Schr\"odinger operators with potentials displaying a spatially stationary structure. In particular, one can in this way explain several common features of these operator families, such as the existence of a non-random spectrum and the relevance of the integrated density of states and (in the case of dimension one) the Lyapunov exponent in the spectral analysis of these operators. 

The general theory of ergodic Schr\"odinger operators (in $\ell^2(\Z)$) has been surveyed in \cite{D17, J07} and presented in detail in \cite{DF22b, DF24}. In addition, there are surveys devoted to the periodic case \cite{K16}, the limit-periodic case \cite{DF20}, the quasi-periodic case \cite{MJ17}, and the random case \cite{K08, S11}.

In addition to providing an umbrella framework in which to study certain subclasses of interest, the class of ergodic Schr\"odinger operators can be studied in its own right. There are several general results, such as Kotani theory \cite{D07, K84, K97, S83}, gap labelling theory \cite{ABD12, B85, DF22a, J86}, and generic spectral singularity \cite{AD05} and continuity \cite{AD25}.

The present paper is written in this spirit and extends the general theory of ergodic Schr\"odinger operators. It shows under minimal assumptions on the base dynamics that various spectral properties arise for a dense set of sampling functions.

To be able to describe these spectral properties in the next section, let us recall some fundamental quantities and results from the general theory of ergodic Schr\"odinger operators. For background and more details, we refer the reader to \cite{D17, DF22b, DF24}.

Given a compact metric space $\Omega$, a homeomorphism $T : \Omega \to \Omega$, a $T$-ergodic Borel probability measure $\mu$, and a continuous sampling function $v : \Omega \to \R$, we can define potentials $V_\omega(n) = v(T^n \omega)$, $\omega \in \Omega$, $n \in \Z$ and Schr\"odinger operators $[H_\omega \psi](n) = \psi(n+1) + \psi(n-1) + V_\omega(n) \psi (n)$ in $\ell^2(\Z)$. The family $\{ H_\omega \}_{\omega \in \Omega}$ is then referred to as an \emph{ergodic family of Schr\"odinger operators} (in $\ell^2(\Z)$). 

First, there is a non-random spectrum $\Sigma_v$ such that $\sigma(H_\omega) = \Sigma_v$ for $\mu$-almost every $\omega \in \Omega$.\footnote{Of course, the non-random spectrum depends on the base dynamics $(\Omega,T,\mu)$ as well, but in this paper the latter will be fixed, while the sampling function $v$ is being varied. We therefore usually make the dependence on $v$ explicit.}

Second, if one defines the \emph{density of states measure} $\nu_v$ on $\R$ as the $\mu$-average of the spectral measure corresponding to the pair $(H_\omega,\delta_0)$, that is,
$\int g \, d\nu_v = \int \langle \delta_0, g(H_\omega) \delta_0 \rangle \, d\mu(\omega)$, the almost sure spectrum coincides with the topological support of the density of states measure,
\begin{equation}\label{e.dossupport}
\Sigma = \supp \, \nu_v.
\end{equation}
The accumulation function of the density of states measure,
$N_v(E) = \int \chi_{(-\infty,E]} \, d\nu_v$, is called the \emph{integrated density of states}.

Third, the Lyapunov exponent measures the average exponential rate of growth of transfer matrices. That is, given $E \in \R$, define $A_E : \Omega \to \SL(2,\R), \; \omega \mapsto \begin{pmatrix} E - v(\omega) & -1 \\ 1 & 0 \end{pmatrix}$
and, for $n \ge 1$, $A_E^n : \Omega \to \SL(2,\R), \; \omega \mapsto A_E(T^{n-1} \omega) \cdots A_E(\omega)$. The \emph{Lyapunov exponent} is then given by $L_v(E) = \lim_{n \to \infty} \frac{1}{n} \int \log \| A_E^n(\omega) \| \, d\mu(\omega)$, and it is connected to the density of states measure via the Thouless formula
\begin{equation}\label{e.thouless}
L_v(E) = \int \log | E' - E | \, d\nu_v(E').
\end{equation}

\subsection{The Main Spectral Results}

The following theorem contains the main results of this paper.

\begin{theorem}\label{t.main2}
Given a compact metric space $\Omega$, a homeomorphism $T : \Omega \to \Omega$, and a non-atomic ergodic Borel probability measure $\mu$, there is a dense set of $v \in C(\Omega,\R)$ such that the following statements hold:

\begin{itemize}

\item[{\rm (a)}] the almost sure spectrum $\Sigma_v$ consists of a finite union of non-degenerate compact intervals;

\item[{\rm (b)}] the integrated density of states $N_v$ is $C^\infty$;

\item[{\rm (c)}] the Lyapunov exponent is $C^\infty$ and obeys $\min \{ L_v(E) : E \in \R \} > 0$.

\end{itemize}
\end{theorem}

\begin{remark}\label{rem.1}
(i) The case of an atomic ergodic Borel probability measure $\mu$ is easily understood. In this case, $\supp \mu$ is finite and indeed given by a single $T$-periodic orbit due to ergodicity. This immediately shows that the associated operators (with $\omega$ restricted to $\supp \mu$) are given by a single periodic Schr\"odinger operator and its translates, so that the standard periodic theory applies. As a consequence, (a) continues to hold; (b) holds in a slightly modified formulation (smoothness of $N_v$ holds away from the boundary of $\Sigma_v$, but not globally); and (c) fails completely (${L_v} \equiv 0$ on $\Sigma_v$).
\\[2mm]
(ii) This theorem can be viewed as follows. Excluding the trivial case of an atomic $\mu$ (which is discussed in item (i) above), we have in complete generality that basic spectral properties known for the Anderson model (with a single-site distribution that has a sufficiently regular density; c.f.~\cite{DF24, KS80, ST85}) hold on a dense set of sampling functions, even if the base dynamics does not display any randomness whatsoever. All the features responsible for such spectral behavior can be encoded in the sampling function, and this can indeed be arranged in any neighborhood of any given sampling function.
\\[2mm]
(iii) Another basic spectral property of the Anderson model is localization, that is, for $\mu$-almost every $\omega \in \Omega$, the operator $H_\omega$ has pure point spectrum with exponentially decaying eigenfunctions. In fact, the statement in part (c), namely the uniform positivity of the Lyapunov exponent, is usually a strong indication of localization and serves as a starting point of a rigorous proof of this spectral feature.\footnote{It should be pointed out, however, that there are Schr\"odinger operators with uniformly positive Lyapunov exponent and continuous spectrum. As far as general consequences go, positive Lyapunov exponents do imply the absence of absolutely continuous spectrum. See \cite{DF22b, DF24} and references therein for more details.} We will devote a forthcoming joint paper with Fillman to proving localization for a dense set of sampling functions with base dynamics $(\Omega,T,\mu)$ as in Theorem~\ref{t.main2} above.
\\[2mm]
(iv) It should be emphasized that none of the statements in (a)--(c) were previously known for a dense set of sampling functions for \emph{any} ergodic base dynamics $(\Omega,T,\mu)$ with $\mu$ non-atomic. In particular, we do not broaden the scope of a known result, but rather exhibit new phenomena, and establish them in the greatest generality possible.
\\[2mm]
(v) One is often interested in spectral properties that hold for a dense $G_\delta$ set of sampling functions. Our theorems on generic singularity \cite{AD05} and continuity \cite{AD25} are of this kind. One can wonder whether the spectral properties in (a)--(c) might be generic in this sense. This, however, is known to be false (e.g., part (a) cannot hold generically due to the complete gap labeling result from \cite{ABD12} as soon as the Schwartzman group is dense in $(0,1)$) or unlikely to be true. In fact, the main point of Theorem~\ref{t.main2} is that non-generic spectral properties may still hold for a dense set of sampling functions.
\end{remark}

\subsection{An Answer to a Question of Walters}

One of the central pillars of ergodic theory \cite{P89, W82} is the Birkhoff ergodic theorem, which states that 
$\lim_{N \to \infty} \frac{1}{N} \, \sum_{n = 0}^{N-1} f (T^n \omega) = \int f \, d\mu$ for $f \in L^1(\Omega,\mu)$ and $\mu$-almost every $\omega \in \Omega$. Here, $(\Omega,\mu)$ is a probability measure space and $T:\Omega\to\Omega$ is ergodic. 

On the other hand, the Kingman subadditive ergodic theorem states that if $f_n : \Omega \to \R$ are integrable and obey $f_{n+m}(\omega) \le f_n(\omega) + f_m(T^n \omega)$, then there exists a constant $f \in [-\infty,\infty)$ such that
\begin{equation}\label{e.kingman}
\lim_{n \to \infty} \frac{1}{n} f_n(\omega) = f
\end{equation}
for $\mu$-almost every $\omega \in \Omega$.

An important application of Kingman's subadditive ergodic theorem yields the existence of Lyapunov exponents for $\mathrm{GL}(N,\R)$-cocycles. Given $A : \Omega \to \mathrm{GL}(N,\R)$ measurable with $\log \|A(\cdot)\|$ integrable, there is a \emph{Lyapunov exponent} $L(A)$ such that for
$A^n(\omega) = A(T^{n-1} \omega) \cdots A(\omega)$, we have
\begin{equation}\label{e.lyapunov}
L(A) = \lim_{n \to \infty} \frac{1}{n} \log \| A^n(\omega) \|
\end{equation}
$\mu$-almost everywhere. 

One may wonder whether almost everywhere convergence may be improved to everywhere uniform convergence. In the additive setting, this question is connected to the concept of unique ergodicity \cite{O52}. If one wants to explore uniform convergence in the subadditive setting, one clearly has to at least assume unique ergodicity. But is unique ergodicity sufficient?

One has to distinguish between the ``trivial case'' of a finitely supported unique ergodic measure and the ``non-trivial case'' of a non-atomic unique ergodic measure. In the former case, uniform convergence is easy to establish. In the latter case, the general subadditive setting is understood: Derriennic and Krengel \cite{DK81} have shown that there are always continuous $f_n$ so that $\{ f_n \}$ is subadditive and uniform convergence fails in \eqref{e.kingman}.

This leaves the question about uniform convergence in \eqref{e.lyapunov} for $\GL(N,\R)$-cocycles over a uniquely ergodic $(\Omega,T)$ with non-atomic unique ergodic measure $\mu$. Clearly, the case $N = 1$ reduces to the additive setting, and failure for one $N$ implies failure for all larger values of $N$. Walters thus asked \cite{W84}:

\begin{question}[Walters]
If $(\Omega,T)$ is uniquely ergodic with a non-atomic unique ergodic Borel probability measure, does there exist a continuous $A : \Omega \to \mathrm{GL}(2,\R)$ for which $\frac{1}{n} \log \|A^n(\omega)\|$ fails to converge uniformly to a constant?
\end{question}

We are able to answer this question in the affirmative:

\begin{theorem}\label{c.walters}
Given a compact metric space $\Omega$, a homeomorphism $T : \Omega \to \Omega$, and a non-atomic ergodic Borel probability measure $\mu$, there exists a non-uniform $\SL(2,\R)$ cocycle. That is, there is a continuous $A : \Omega \to \SL(2,\R)$ for which the matrix products $A_n(\omega) = A(T^{n-1} \omega) \cdots A(\omega)$ are such that $\frac1n \log \|A_n(\cdot)\|$ fails to converge uniformly to the associated Lyapunov exponent $L(A) = \lim_{n \to \infty} \frac1n \int \log \|A_n(\omega)\| \, d\mu(\omega)$.
\end{theorem}

\section{Proofs}

\begin{definition}\label{def.kernel}
Fix once and for all a $C^\infty$ function $s : \R \to [0,\infty)$ with $\supp s = [-1,1]$ and $\int s(x) \, dx = 1$. For $\varepsilon > 0$, consider the rescaled function $s_\varepsilon (x) = \varepsilon^{-1} s(\varepsilon^{-1}x)$ with $\supp s_\varepsilon = [-\varepsilon,\varepsilon]$ and $\int s_\varepsilon (x) \, dx = 1$. Given a Borel probability measure $\nu$ on $\R$, its convolution with $s_\varepsilon$ is denoted by $S_\varepsilon \nu$:
$$
(S_\varepsilon \nu)(g) = \iint g(x+y) s_\varepsilon(x) \, dx \, d\nu(y).
$$
\end{definition}

\begin{remark}\label{r.smoothing}
We say that $\nu$ is $C^\infty$ if it is absolutely continuous with respect to the Lebesgue measure on $\R$ and its Radon-Nikodym derivative is $C^\infty$. It is easy to see that $S_\varepsilon \nu$ is $C^\infty$ for every $\varepsilon > 0$. 
Specifically, $S_\varepsilon \nu$ is absolutely continuous with respect to Lebesgue measure with Radon-Nikodym derivative given by $\int s_\varepsilon(x-y) \, d\nu(y)$, which is smooth. 
\end{remark}

\begin{lemma}\label{l.smoothinglemma}
Suppose $\nu_n, \nu$ are probability measures on $\R$, all supported within a common compact set, with $\nu_n \to \nu$ in the weak-$*$ sense. 

{\rm (a)} For every $\varepsilon > 0$, we have $S_\varepsilon \nu_n \to S_\varepsilon \nu$ in $C^\infty$.

{\rm (b)} If $\nu$ is $C^\infty$, then $S_{\varepsilon} \nu \to \nu$
 in $C^\infty$ as $\varepsilon \to 0$.

{\rm (c)} If $\nu$ is $C^\infty$, then there exists a sequence $\varepsilon_n \to 0$ such that 
$$
\lim_{n \to \infty} \sup_{n' \ge n} \mathrm{dist}_{C^\infty} ( S_{\varepsilon_n} \nu_{n'}, \nu ) = 0.
$$
\end{lemma}

\begin{remark}\label{r.cinftyconv}
Denoting by $K$ a compact set that supports all measures in question, convergence in $C^\infty$ of the associated functions is then understood as the uniform convergence on $K$ of all derivatives. In other words, $f_n \to f$ in $C^\infty$ if and only if
$$
\mathrm{dist}_{C^\infty}(f_n,f) := \sum_{j = 0}^\infty 2^{-j} \min \{ \|f_n^{(j)} - f^{(j)}\|_{\infty,K} , 1 \} \to 0 \text{ as } n \to \infty.
$$
\end{remark}

\begin{proof}
(a) By assumption,
$$
\forall f \in C_b(\R) : \int f \, d\nu_n \to \int f \, d\nu \text{ as } n \to \infty.
$$

Moreover, as pointed out above, $S_\varepsilon \nu_n$ and $S_\varepsilon \nu$ are $C^\infty$ with density given by the convolution of $s_{\varepsilon}$ with the measure in question.

Thus it suffices to show that the densities converge in $C^\infty$. Combining the two ingredients, we observe for $k \in \Z_+$ that
\begin{align*}
\frac{d^k}{dx^k} \int s_\varepsilon(x-y) \, d\nu_n(y) & = \int s_\varepsilon^{(k)}(x-y) \, d\nu_n(y) \\
& \to \int s_\varepsilon^{(k)}(x-y) \, d\nu(y) \\
& = \frac{d^k}{dx^k} \int s_\varepsilon(x-y) \, d\nu(y)
\end{align*}
The convergence is easily seen to be uniform in $x$.

\medskip

(b) By assumption, there is a $C^\infty$ function $h$ such that
$\int g \, d\nu = \int g(x) h(x) \, dx$. This in turn gives that
$$
\int g \, d(S_\varepsilon \nu) = \iint g(x) s_\varepsilon(x-y) \, d\nu(y) \, dx = \int g(x) \left( \int s_\varepsilon(x-y) h(y) \, dy \right) dx,
$$
and hence the statement follows from the fact that $\int s_\varepsilon(x-y) h(y) \, dy$ converges to $h$ in $C^\infty$, which is a well-known mollifier fact.

\medskip

(c) Our goal is to find $\varepsilon_n \to 0$ such that $S_{\varepsilon_n} \nu_n \to \nu$ in $C^\infty$. 
Given $m \in \N$, we can choose $\varepsilon^{(m)} > 0$ small enough so that
$$
\mathrm{dist}_{C^\infty} ( S_{\varepsilon^{(m)}} \nu, \nu ) < \frac1m
$$
by part (b). Of course this choice can be made in such a way that $\varepsilon^{(m)} \to 0$ as $m \to \infty$. Moreover, by part (a), there is $n^{(m)} \in \N$ such that
$$
\sup_{n \ge n^{(m)}} \mathrm{dist}_{C^\infty} ( S_{\varepsilon^{(m)}} \nu_{n}, S_{\varepsilon^{(m)}} \nu ) < \frac1m.
$$
Again, this choice can be made in such a way that the sequence $\{ n^{(m)} \}$ is strictly increasing, $n^{(1)} < n^{(2)} < n^{(3)} < \cdots$, and hence diverging.

With these choices we now define
$$
\varepsilon_n := \varepsilon^{(m(n))}, \quad m(n) := \max \{ m : n^{(m)} \le n \}.
$$
Then, clearly $m(n) \to \infty$ as $n \to \infty$, and hence $\varepsilon_n \to 0$ as $n \to \infty$. Moreover,
\begin{align*}
\sup_{n' \ge n} \mathrm{dist}_{C^\infty} ( S_{\varepsilon_n} \nu_{n'}, \nu ) & \le \sup_{n' \ge n} \mathrm{dist}_{C^\infty} ( S_{\varepsilon_n} \nu_{n'}, S_{\varepsilon_n} \nu ) + \mathrm{dist}_{C^\infty} ( S_{\varepsilon_n} \nu , \nu ) \\
&  < \frac{1}{m(n)} + \frac{1}{m(n)},
\end{align*}
which goes to zero as $n \to \infty$ by the discussion above.
\end{proof}

\begin{proof}[Proof of Theorem~\ref{t.main2}]
Let $\tilde v \in C(\Omega,\R)$ and $\varepsilon > 0$ be given. Our goal is to find $v \in C(\Omega,\R)$ with 
\begin{equation}\label{e.mainproof1}
\|v - \tilde v\|_\infty < \varepsilon 
\end{equation}
for which the properties in (a)--(c) hold. 

The heart of the proof of deals with the density of states measure since the spectrum and the Lyapunov exponent can be directly derived from it (via the topological support formula \eqref{e.dossupport} and the Thouless formula \eqref{e.thouless}, respectively).

Let us emphasize that all measures below will have support inside a common compact set (e.g., the set $[-M,M]$, where $M = 2 + \|\tilde v\|_\infty + \varepsilon$). In particular, any lower bound on the length of a connected component of the topological support of any of these measures will yield an upper bound on the number of components in a uniform way.

We begin with a preparatory step. Choose $\varepsilon_n > 0$, $n \ge 0$, such that
\begin{equation}\label{e.mainproof9a}
\sum_{n = 0}^\infty \varepsilon_n < \varepsilon.
\end{equation}
These numbers will serve as allowable adjustments in our inductive construction below, and hence \eqref{e.mainproof9a} will ensure \eqref{e.mainproof1}. By \cite[Theorem~1]{AD05} (which applies since $\mu$ was assumed to be non-atomic) we may approximate $\tilde v$ with $v_0$ such that
\begin{equation}\label{e.mainproof2}
\|v - v_0\|_\infty < \varepsilon_0
\end{equation}
and 
\begin{equation}\label{e.mainproof3}
L_{v_0}(E) > 0 \text{ for Lebesgue almost every } E \in \R.
\end{equation}
Consider the smoothing $S_{\varepsilon^{(0)}} \nu_{v_0}$ for some 
\begin{equation}\label{e.mainproof9b}
0 < \varepsilon^{(0)} < \varepsilon_1.
\end{equation}
By Remark~\ref{r.smoothing}.(ii), $S_{\varepsilon^{(0)}} \nu_{v_0}$ is $C^\infty$ and by the way it is constructed, $\supp S_{\varepsilon^{(0)}} \nu_{v_0}$ consists of a finite union of compact intervals of length at least $2\varepsilon^{(0)}$. Moreover, by \eqref{e.mainproof3} we have\footnote{To derive this, recall that $L_{v_0}(E) \to \infty$ as $|E| \to \infty$.}
\begin{equation}\label{e.mainproof6}
L^{(0)} := \inf \{ S_{\varepsilon^{(0)}} L_{v_0}(E) : E \in \R \} > 0.
\end{equation}
Choose $\ell_n > 0$, $n \ge 1$, such that
\begin{equation}\label{e.mainproof6a}
L^{(0)} - \sum_{n = 1}^\infty \ell_n > 0.
\end{equation}
These numbers will also serve as allowable adjustments in our inductive construction below, and hence \eqref{e.mainproof6a} will guarantee the positivity of the Lyapunov exponent in the limit.

Next, for $n \ge 1$, we will inductively choose  smoothing sizes $0 < \varepsilon^{(n)} < \varepsilon_{n+1}$ and sampling functions $v_{n} \in C(\Omega,\R)$ 
with (see Remark~\ref{r.cinftyconv} for the definition of $\mathrm{dist}_{C^\infty}$)
\begin{equation}\label{e.mainproof11}
\|v_{n} - v_{n-1}\|_\infty < \varepsilon_n,
\end{equation}
\begin{equation}\label{e.mainproof11a}
I \text{ is a connected component of } \supp S_{\varepsilon^{(n)}} \nu_{v_n} \; \Rightarrow \; |I| \ge 2 \varepsilon^{(0)},
\end{equation}
\begin{equation}\label{e.mainproof11c}
\mathrm{dist}_{C^\infty}( S_{\varepsilon^{(n-1)}} \nu_{v_{n-1}}, S_{\varepsilon^{(n)}} \nu_{v_n}) < \varepsilon_n,
\end{equation}
\begin{equation}\label{e.mainproof11b}
\inf \{ S_{\varepsilon^{(n)}} L_{v_n}(E) : E \in \R \} \ge L^{(0)} - \sum_{k = 1}^n \ell_k.
\end{equation}

By \eqref{e.mainproof9a}, \eqref{e.mainproof2},  and \eqref{e.mainproof11}, it follows that the limit $v := \lim_{n \to \infty} v_n$ exists in $C(\Omega,\R)$ and obeys \eqref{e.mainproof1}. Moreover, as $0 < \varepsilon^{(n)} < \varepsilon_{n+1}$  and the $\varepsilon_n$'s are summable, it follows that the properties (a)--(c) in Theorem~\ref{t.main2} hold for $v$. Thus it remains to carry out the inductive procedure.

\medskip

Let $n \ge 1$ be given and assume that $\varepsilon^{(0)}, \ldots, \varepsilon^{(n-1)}$ and $v_0, \ldots, v_{n-1}$ have already been chosen with the desired properties. Let us realize $S_{\varepsilon^{(n-1)}} \nu_{v_{n-1}}$ as the density of states measure of a suitable (non-ergodic) family of Schr\"odinger operators, defined on the larger space
\begin{equation}\label{e.mainproof4}
\tilde \Omega := \Omega \times [-\varepsilon^{(n-1)},\varepsilon^{(n-1)}],
\end{equation}
equipped with the homeomorphism $\tilde T := T \times \mathrm{id}$, where $\mathrm{id}$ is the identity map on $[-\varepsilon^{(n-1)},\varepsilon^{(n-1)}]$, and the measure $\tilde \mu:= \mu \times s_{\varepsilon^{(n-1)}}$, with the absolutely continuous measure $s_{\varepsilon^{(n-1)}}$ on $[-\varepsilon^{(n-1)},\varepsilon^{(n-1)}]$ with density $s_{\varepsilon^{(n-1)}}$ from Definition~\ref{def.kernel} (we use the same symbol for the density and the measure). Of course $(\tilde \Omega, \tilde T, \tilde \mu)$ is not ergodic, but we can effectively view the operator family resulting via the sampling function
\begin{equation}\label{e.potentialshift}
\tilde v_{n-1} \left( \omega, s \right) = v_{n-1}(\omega) + s
\end{equation}
as a direct integral of ergodic operator families (obtained by shifting the original potentials by constants $s$ chosen from the interval $[-\varepsilon^{(n-1)},\varepsilon^{(n-1)}]$) with respect to the measure $s_{\varepsilon^{(n-1)}}$. Observe that, by this very description, the density of states measure of the enlarged family of operators is indeed given by $S_{\varepsilon^{(n-1)}} \nu_{v_{n-1}}$. Concretely, the potential shift \eqref{e.potentialshift} results in a translation of the spectral measure corresponding to $(H_\omega, \delta_0)$ by $s$, and hence the $\mu$ integration results in a shift of the density of states measure by $s$, after which the integration in the second variable produces the desired convolution.

Via the Thouless formula (or even directly since it is just an energy shift) we can see that the Lyapunov exponent of the enlarged family is given by $S_{\varepsilon^{(n-1)}} L_{v_{n-1}}$. In particular, it is smooth and uniformly positive.

Next we want to model the same behavior on the smaller space $\Omega$ and choose a sampling function $v_n \in C(\Omega,\R)$ whose density of states measure $\nu_{v_n}$ is very close to $S_{\varepsilon^{(n-1)}} \nu_{v_{n-1}}$. 

To this end, as $\mu$ was assumed to be non-atomic, by Kakutani–Rokhlin we may choose a tower of arbitrary height that exhausts $\Omega$ up to arbitrarily small measure. That is, for any positive integer $K$, there is a collection of mutually disjoint compact subsets $\Omega^{(K)}_1, \ldots, \Omega^{(K)}_K$ of $\Omega$ such that $T(\Omega^{(K)}_j) = \Omega^{(K)}_{j+1}$, $j = 1, \ldots, K-1$ and such that the measure of the complement of the tower,
$$
\mu \left( \Omega \setminus \bigcup_{j=1}^K \Omega^{(K)}_j \right), 
$$
approaches $0$ as $K \to \infty$.

We subdivide (or, more accurately, pack) $\Omega^{(K)}_1$ further into a large number of mutually disjoint compact subsets $\Omega^{(K,K')}_{1,1}, \ldots, \Omega^{(K,K')}_{1,K'}$ of equal measure, so that their union exhausts $\Omega^{(K)}_1$ up to arbitrarily small measure for $K'$ sufficiently large. Mapping forward with $T^{j-1}$, this provides a subdivision/packing of $\Omega^{(K)}_j$ by sets $\Omega^{(K,K')}_{j,1}, \ldots, \Omega^{(K,K')}_{j,K'}$ for each $j = 2,\ldots,K$. 

For $\ell = 1,\ldots, K'$ we choose values $s_\ell$ that follow from a suitable discretization of $([-\varepsilon^{(n-1)},\varepsilon^{(n-1)}],  s_{\varepsilon^{(n-1)}})$, that is, the values are chosen from a finite  relatively dense subset of $[-\varepsilon^{(n-1)},\varepsilon^{(n-1)}]$ with relative frequency respecting $s_{\varepsilon^{(n-1)}}$. We now define $v_n \in C(\Omega,\R)$ as follows: for $\omega \in \Omega^{(K,K')}_{j,\ell}$, we set $v_n(\omega) = v_{n-1}(\omega) + s_\ell$, and outside $\bigcup_{j,\ell} \Omega^{(K,K')}_{j,\ell}$, we use the Tietze extension theorem to define $v_n$ so that it is continuous on $\Omega$ and obeys
\begin{equation}\label{e.mainproof7}
\| v_n - v_{n-1} \|_\infty \le \varepsilon^{(n-1)} < \varepsilon_n.
\end{equation}

In this construction, $v_n$ still depends on $K,K'$, which are suitable large positive integers. Making this dependence explicit, denote the associated density of states measure by $\nu_{v_n^{(K,K')}}$ and  notice that the choices above can be arranged in such a way that $\nu_{v_n^{(K,K')}}$ converges to $S_{\varepsilon^{(n-1)}} \nu_{n-1}$ in the weak-$*$ sense as $K,K' \to \infty$. Applying Lemma~\ref{l.smoothinglemma}, we can therefore  choose first
$$
0 < \varepsilon^{(n)} < \varepsilon_{n+1}
$$ 
and then $K,K'$ sufficiently large such that for $v_n := v_n^{(K,K')}$, we have \eqref{e.mainproof11}--\eqref{e.mainproof11b}. Indeed, \eqref{e.mainproof11} holds by construction and by part (c) of Lemma~\ref{l.smoothinglemma}, taking $\varepsilon^{(n)}$ sufficiently small and $K,K'$ sufficiently large, we can ensure that \eqref{e.mainproof11c} and (via the Thouless formula) \eqref{e.mainproof11b} both hold. Finally, to see that \eqref{e.mainproof11a} holds, note that $\supp \nu_{v_n}$ becomes $\varepsilon^{(n)}$-dense in $\supp S_{\varepsilon^{(n-1)}} \nu_{v_{n-1}}$ if $K,K'$ are chosen large enough, so that each component of $\supp  S_{\varepsilon^{(n)}} \nu_{v_n}$ must contain a component of  $\supp S_{\varepsilon^{(n-1)}} \nu_{v_{n-1}}$. 

This establishes the desired properties at level $n$ and completes the discussion of the inductive procedure.
\end{proof}

\begin{proof}[Proof of Theorem~\ref{c.walters}]
Given a compact metric space $\Omega$, a homeomorphism $T : \Omega \to \Omega$, and a non-atomic ergodic Borel probability measure $\mu$, choose a $v \in C(\Omega,\R)$ from the dense set for which the spectral properties (a)--(d) listed in Theorem~\ref{t.main2} hold. Property~(c) implies that there is an $E$ (any $E \in \Sigma_v$ can be chosen) for which the cocycle $A^E$ is non-uniform, due to Johnson's theorem \cite{J86}.\footnote{See also Furman \cite{F97} and Lenz \cite{L02, L04} for related results.} 
\end{proof}

\begin{remark}
One may wonder whether the generic almost everywhere positivity result for the Lyapunov exponent established in \cite{AD05} is already sufficient for (or at least strongly indicative of) an affirmative answer to Walters' question. The issue is, however, that even if the Lyapunov exponent is almost everywhere positive, it does not follow that it must be positive somewhere on the spectrum -- the spectrum could have zero Lebesgue measure for the sampling functions in question. Indeed, this is precisely how \cite{DL06} establish zero-measure spectrum for Boshernitzan subshifts and locally constant sampling functions (which are dense in the continuous functions over such subshifts), namely by showing that the Boshernitzan condition is incompatible with the non-uniformity of such cocycles.
\end{remark}

\end{document}